\documentclass[envcountsame]{llncs}

\usepackage{amssymb}
\usepackage{mathcomp}
\usepackage{relsize}
\usepackage{amsmath}
\usepackage{mathrsfs}
\usepackage{color}
\usepackage{mathtools,amssymb}
\usepackage[utf8]{inputenc}
\inputencoding{utf8} 
\usepackage[english]{babel}
\usepackage{enumitem}
\setlist{leftmargin=*}
\usepackage{etoolbox}
\usepackage{nicefrac}
\newtoggle{arxiv}
\toggletrue{arxiv}

\title{Continuity of the Shafer-Vovk-Ville Operator}
\author{Natan T'Joens \and Gert de Cooman \and Jasper De Bock}
\institute{Ghent University, ELIS, SYSTeMS\\\email{\{natan.tjoens,gert.decooman,jasper.debock\}@ugent.be}}

\newcommand{\reals}{\bbbr}
\newcommand{\extreals}{\overline{\bbbr}}
\newcommand{\posreals}{\reals_{> 0}}
\newcommand{\nnegreals}{\reals_{\geq 0}}
\newcommand{\nats}{\bbbn}
\newcommand{\natz}{\bbbn_{0}}

\newcommand{\indica}[1]{\mathbb{I}_{#1}}

\newcommand{\upprev}{\overline{\mathrm{E}}}
\newcommand{\upprevvovk}{\overline{\mathrm{E}}_\mathrm{V}}
\newcommand{\lowprev}{\underline{\mathrm{E}}}
\newcommand{\lowprevvovk}{\underline{\mathrm{E}}_\mathrm{V}}

\newcommand{\lupprev}{\overline{\mathrm{Q}}}

\newcommand{\martingale}{\mathscr{M}}
\newcommand{\setofsupmart}{\overline{\mathbb{M}}}
\newcommand{\setofsupmartb}{\overline{\mathbb{M}}_\mathrm{b}}

\newcommand{\process}{\mathscr{L}}

\newcommand{\situations}{\mathscr{X}^\ast}
\newcommand{\statespace}{\mathscr{X}}
\newcommand{\statespaceseq}[2]{\mathscr{X}_{#1:#2}}

\newcommand{\samplespace}{\Omega}

\newcommand{\setofgambles}{\mathbb{G}}

\newcommand{\setofextvariables}{\overline{\mathbb{V}}}

\spnewtheorem{example}{Example}{\itshape}{\rmfamily}

\spnewtheorem{definition}{Definition}{\bfseries}{\rmfamily}

\begin{document}
\maketitle

\begin{abstract}
Kolmogorov's axiomatic framework is the best-known approach to describing probabilities and, due to its use of the Lebesgue integral, leads to remarkably strong continuity properties.
However, it relies on the specification of a probability measure on all measurable events.
The game-theoretic framework proposed by Shafer and Vovk does without this restriction.
They define global upper expectation operators using local betting options. 
We study the continuity properties of these more general operators. 
We prove that they are continuous with respect to upward convergence and show that this is not the case for downward convergence.
We also prove a version of Fatou's Lemma in this more general context.
Finally, we prove their continuity with respect to point-wise limits of two-sided cuts. 
\end{abstract}


\section{Introduction}
\label{sec:Introduction}
The most common approach to probability theory is the measure-theoretic framework that originates in Kolmogorov's work~\cite{kolmogorov1933grundbegriffe}. 
Its popularity is largely due to its interpretational neutrality and the elegant mathematical properties resulting from the use of measure theory.
However, this framework requires the definition of a probability measure on all measurable events.
Although this is often overlooked, it presents a major drawback, because the actual specification of these probabilities is far from trivial in many practical applications.
Hence, the mathematical results are elegant, but the underlying assumptions are very strong.

For dealing with stochastic processes, a more general and intuitive approach was proposed by Shafer and Vovk~\cite{Shafer:2005wx}. 
Their so-called game-theoretic framework is based on the idea of a supermartingale: a specific way to gamble on the successive outcomes of the process.

In our present contribution, we study the continuity properties of the upper (and therefore also lower) expectation operators that appear in this framework.
Our main results are that they are continuous with respect to upward, but not downward, convergence of uniformly bounded below sequences, and continuous with respect to particular limits of two-sided cuts.
From our upward convergence result, we also derive a generalised version of Fatou's Lemma.

\iftoggle{arxiv}{
	In order to facilitate the reading, proofs and intermediate results are relegated to the appendix.
}
{
	Due to length constraints, proofs and intermediate results are relegated to the appendix of an extended online version of this paper, available on ArXiv~\cite{TJoens2018Continuity}
}

\section{Preliminaries}
We denote the set of all natural numbers, without $0$, by $\nats$, and let $\natz \coloneqq \nats \cup \{0\}$. 
The set of extended real numbers is denoted by $\extreals \coloneqq \reals \cup \{+ \infty, - \infty\}$.  
The set of positive real numbers is denoted by $\posreals$ and the set of non-negative real numbers by $\nnegreals$.

We consider sequences of uncertain states $X_1, X_2, ..., X_n , ...$ where the state $X_k$ at each discrete time $k \in \nats$ takes a value in some non-empty finite set $\statespace{}$, called the \emph{state space}.
We call any $x_{1:n} \coloneqq (x_1,...,x_n) \in \statespace{}_{1:n} \coloneqq \statespace{}^n$, for $n \in \natz$, a \emph{situation} and we denote the set of all situations by $\statespace{}^\ast \coloneqq \cup_{n \in \natz{}} \mathscr{X}_{1:n}$.
So any finite string of possible values for a sequence of consecutive states is called a situation.
In particular, the unique empty string $x_{1:0}$, denoted by $\Box$, is called the \emph{initial situation}: $\statespace{}_{1:0} \coloneqq \{\Box\}$.

An infinite sequence of state values $\omega$ is called a \emph{path} and the set of all paths is called the \emph{sample space} $\samplespace{} \coloneqq \statespace{}^\nats$.
For any path $\omega \in \samplespace$, the initial sequence that consists of its first $n$ state values is a situation in $\statespaceseq{1}{n}$ that is denoted by $\omega^n$. 
The $n$-th state is denoted by $\omega_n \in \statespace$.

\section{Game-Theoretic Upper Expectations}
\label{sec: Game-theoretic lower and upper expectations}
In order to deal with stochastic processes mathematically, we use variables. 
A global variable, or simply \emph{variable}, $f$ is a map on the set $\Omega$ of all paths.
An \emph{(extended) real variable} $f$ associates an (extended) real number $f(\omega)$ with any path $\omega$.
The set of all extended real variables is denoted by $\setofextvariables$. 
For any natural $k \leq \ell$, we use $X_{k:\ell}$ to denote the variable that, for every path $\omega$, returns the tuple $X_{k:\ell}(\omega)\coloneqq(\omega_k,...,\omega_\ell)$.
As such, the state $X_k = X_{k:k}$ at any discrete time $k$ can also be regarded as a variable.

A collection of paths $A \subseteq \samplespace$ is called an \emph{event}.
The \emph{indicator} $\indica{A}$ of an event $A$ is defined as the variable that assumes the value $1$ on $A$ and $0$ elsewhere.
With any situation $x_{1:n}$, we associate the \emph{cylinder event} $\Gamma(x_{1:n})\coloneqq \{\omega \in \samplespace\colon \omega^n = x_{1:n}\}$: the set of all paths $\omega \in \samplespace$ that go through the situation $x_{1:n}$.
For a given $n \in \natz$, we call a variable $f$ \emph{$n$-measurable} if it is constant on the cylinder events $\Gamma(x_{1:n})$ for all $x_{1:n} \in \statespaceseq{1}{n}$, that is, if we can write $f = \tilde{f} \circ X_{1:n} = \tilde{f}(X_{1:n})$ for some map $\tilde{f}$ on $\statespace{}^n$. 
We will then use the notation $f(x_{1:n})$ for its constant value $f(\omega)$ on all paths $\omega \in \Gamma(x_{1:n})$.

Variables are inherently uncertain objects and, therefore, we need a way to model this uncertainty.
We will do this by means of upper and lower expectations, which requires the introduction of gambles.
 
For any non-empty set $\mathscr{Y}$, we define a \emph{gamble} $f$ on $\mathscr{Y}$ as a bounded real map on $\mathscr{Y}$.
It is then typically interpreted as an uncertain reward $f(y)$ when the outcome of some `experiment', assuming values in $\mathscr{Y}$, is $y \in \mathscr{Y}$.
The set of all gambles on $\mathscr{Y}$ is denoted by $\mathbb{G}(\mathscr{Y})$.
In particular, a gamble on $\samplespace{}$ is a bounded real variable.
When $\mathscr{Y} = \statespace{}$, we call the gamble $f$ a \emph{local gamble}.

A coherent upper expectation $\upprev$ on the set $\setofgambles(\mathscr{Y})$ is defined as a real functional on $\setofgambles(\mathscr{Y})$ that satisfies the following \emph{coherence axioms}~\cite[2.6.1]{Walley:1991vk}:
\begin{enumerate}[ref={\upshape E}\arabic*,label={\upshape E}\arabic*.]
\item \label{coherence 1} 
$\upprev(f) \leq \sup f$ for all $f \in \setofgambles(\mathscr{Y})$;{}
\item \label{coherence 2} 
$\upprev(f+g) \leq \upprev(f) + \upprev(g)$ for all $f,g \in \setofgambles(\mathscr{Y})$;
\item \label{coherence 3}
$\upprev(\lambda f) = \lambda \upprev(f)$ for all $f \in \setofgambles(\mathscr{Y})$ and real $\lambda \geq 0$.
\end{enumerate}
$\upprev(f)$ can be interpreted as some subject's minimum selling price for the gamble $f\in \setofgambles(\mathscr{Y})$ on $\mathscr{Y}$. 
Alternatively, one can also consider the conjugate lower expectation, defined by $\lowprev(f)\coloneqq-\upprev(-f)$ for all $f\in\setofgambles(\mathscr{Y})$. 
It clearly suffices to focus on only one of the two functionals. 
We will work with upper expectations.

In an \emph{imprecise probability tree} we attach to each situation $x_{1:n} \in \situations$ a \emph{local} probability model characterised by a coherent upper expectation $\lupprev(\cdot \vert x_{1:n})$ on the set $\setofgambles(\statespace)$ of all \emph{local} gambles on the next state $X_{n+1}$.
These local probability models $\lupprev(\cdot \vert x_{1:n})$ are usually known, as, in most practical cases, they can be elicited fairly easily from a subject or learned from data.
They express a subject's beliefs or knowledge about the next possible state.
However, gathering information or eliciting beliefs about a variable that depends on multiple states or even entire paths is not that straightforward. 
Therefore, the question arises how we can extend the local probability models (on single states) towards global probability models (on entire paths).

To answer this question, we first need to introduce the concepts of a process and a gamble process.
A \emph{process} $\mathscr{L}$ is a map defined on $\situations$. 
A real process associates a real number $\mathscr{L}(s) \in \reals$ with any situation $s \in \situations$. 
A real process is called positive (non-negative) if it is positive (non-negative) in every situation.
With any real process $\mathscr{L}$ we associate a sequence of $n$-measurable gambles $\{\mathscr{L}_n\}_{n \in \natz{}}$: for all $n \in \natz$, we let $\mathscr{L}_n(\omega) \coloneqq \mathscr{L}(\omega^n)$ for all $\omega \in \samplespace$ or, equivalently, $\mathscr{L}_n \coloneqq \mathscr{L} \circ X_{1:n} = \mathscr{L}(X_{1:n})$.
A \emph{gamble process} $\mathscr{D}$ is a process that associates with any situation $x_{1:n} \in \situations{}$ a local gamble $\mathscr{D}(x_{1:n}) \in \setofgambles{}(\statespace{})$. 
With any real process $\mathscr{L}$, we can associate a gamble process $\Delta \mathscr{L}$, called its \emph{process difference}. 
For any situation $x_{1:n}$ the corresponding gamble $\Delta \mathscr{L}(x_{1:n}) \in \setofgambles(\statespace)$ is defined by
\begin{equation*}
\Delta \mathscr{L}(x_{1:n})(x_{n+1})\coloneqq\mathscr{L}(x_{1:n+1}) - \mathscr{L}(x_{1:n}) \text{ for all } x_{n+1} \in \statespace.
\end{equation*}

\noindent
We will also use the extended real variables $\liminf \mathscr{L} \in \setofextvariables{}$ and $\limsup \mathscr{L} \in \setofextvariables{}$, defined by:
\begin{align*}
\liminf \mathscr{L}(\omega)\coloneqq\liminf_{n \to +\infty} \mathscr{L}_n(\omega)
\text{~~and~~}
\limsup \mathscr{L}(\omega)\coloneqq\limsup_{n \to +\infty} \mathscr{L}_n(\omega) 
\end{align*}
for all $\omega \in \Omega$.
If $\liminf \mathscr{L} = \limsup \mathscr{L}$, we denote their common value by $\lim \mathscr{L}$.

For a \emph{given} imprecise probability tree, a \emph{supermartingale} $\martingale$ is a \emph{real} process for which the process difference $\Delta \martingale$ has a non-positive local upper expectation everywhere: $\lupprev(\Delta \martingale(x_{1:n}) \vert x_{1:n}) \leq 0$ for all $x_{1:n} \in \situations$. 
In other words, a supermartingale is a process that, according to the local probability models, is expected to decrease.
The concept originates in the following `game-theoretic' argument.
Suppose that a forecaster sets minimum selling prices for every gamble $f$ on the next state $X_{n+1}$, i.e. he defines $\lupprev(f \vert x_{1:n})$. 
Non-positive minimum selling prices imply that he is willing to give away these gambles. Suppose now that you take him up on his commitments. 
The gambles available to you are then exactly the ones with $\lupprev(f \vert x_{1:n}) \leq 0$. 
Choosing such an available gamble in every situation $x_{1:n} \in \situations$ essentially defines a supermartingale. 
In this way, we can interpret a supermartingale as a strategy for gambling against a forecaster. 

We define supermartingales here as \emph{real} processes, whereas Shafer and Vovk define them as extended real processes~\cite{shafer2011levy}.
For any situation $s \in \situations{}$, such an extended real process allows the possibility for $\Delta \martingale{}(s)$ to be an extended real function on $\statespace{}$. 
However, it is not immediately obvious to us how to give a behavioural meaning to such extended real process differences, and we therefore prefer to define supermartingales as real processes whose differences are gambles.

We denote the set of all supermartingales \emph{for a given imprecise probability tree} by $\setofsupmart$. 
The set of all bounded below supermartingales is denoted by $\setofsupmartb$. 

We are now ready to introduce the game-theoretic upper expectation.
\begin{definition}\label{def:upperexpectation}
The upper expectation $\upprevvovk(\cdot \vert \cdot)$ 
is defined by
\begin{align}\label{upprev2}
\upprevvovk (f \vert s)\coloneqq \inf \big\{ \martingale(s) \colon \martingale \in \setofsupmartb \text{ and } (\forall \omega \in \Gamma(s)) \liminf \martingale(\omega) \geq f(\omega) \big\},
\end{align}
for all extended real variables $f \in \setofextvariables$ and all $s \in \situations$.
\end{definition}
This definition can be interpreted in the following way: the upper expectation of a variable $f$ when in a situation $s$, is the infimum starting capital in the situation $s$ such that, by using the available gambles from $s$ onwards, we are able to end up with a capital that dominates $f$, \emph{no matter the path through $s$ taken by the process}. 
Importantly, these upper expectations for global variables are defined in terms of supermartingales, and therefore, derived directly from the local models. 
Moreover, due to~\cite[Prop. 8.8]{Shafer:2005wx}, for every situation $s$, the restriction of $\upprevvovk(\cdot\vert s)$ to $\setofgambles{}(\samplespace{})$ satisfies the coherence axioms \ref{coherence 1}--\ref{coherence 3}. 

Observe that in defining these global upper expectations, we consider supermartingales that are bounded below, because as is shown in~\cite[Example~1]{DECOOMAN201618}, for extended real variables, the use of unbounded supermartingales leads to undesirable results, whereas Definition~\ref{def:upperexpectation} does not.

In the remainder of this contribution, we restrict our attention to upper expectations conditional on the initial situation $\Box$ and use the notation $\upprevvovk(f) \coloneqq \upprevvovk(f \vert \Box)$. 
This facilitates the reading and makes the paper conceptually easier.
That being said, we stress that all our arguments are easily extendible to upper expectations conditional on a general situation $s \in \situations{}$.

\section{Continuity with Respect to Upward Convergence}
\label{sec: Continuity properties}
The relevance of continuity properties for (upper) expectation functionals is evident. 
Not only do they provide the mathematical theory with elegance, they also enhance its practical scope. 
The continuity of the Lebesgue integral, for instance, is one of the reasons why it is the integral of choice for computing expected values associated with a probability measure.
Continuity properties provide constructive ways to calculate expectations that otherwise would be difficult or even impossible to calculate numerically.
For example, calculating the upper expectation $\upprevvovk(f)$ of an extended real variable $f$ directly is typically practically impossible if it depends on an infinite number of states. 
However, if we can find a sequence of simpler functions $\{f_n\}_{n \in \natz}$ that converges in some way to $f$, such that the upper expectation $\upprevvovk$ is continuous with respect to this convergence, then we can easily approximate $\upprevvovk(f)$ by $\upprevvovk(f_n)$, provided $n$ is large enough. 
If we can find a sequence for which moreover the individual $\upprevvovk(f_n)$ can be calculated directly, we obtain a practical method for calculating $\upprevvovk(f)$.
Unfortunately, it appears little is known at present about the continuity properties of the functional $\upprevvovk$; we aim to remedy this situation here.

It is well-known that every coherent upper expectation $\upprev$ is continuous with respect to uniform convergence~\cite[p.63]{troffaes2014}: if a sequence of gambles $\{f_n\}_{n \in \natz}$ converges uniformly to a gamble $f$, meaning that $\lim_{n \to +\infty} \sup \{\vert f - f_n \vert\} = 0$, then $\lim_{n \to +\infty} \upprev(f_n) = \upprev(f)$. 
Hence, since the restriction of $\upprevvovk$ to $\mathbb{G}(\samplespace{})$ is a coherent upper expectation, it is continuous with respect to the uniform convergence of gambles on $\samplespace{}$.
This type of continuity is however fairly weak, because the condition of uniform convergence is a very strong one.
Moreover, continuity with respect to pointwise convergence is not directly implied by mere coherence~\cite[p.63]{troffaes2014}.
The following example demonstrates that, also for the upper expectation operator $\upprevvovk$ we are focussing on here, continuity with respect to pointwise convergence, and downward convergence in particular, may fail. 

\begin{example}
Consider, in each situation $x_{1:n} \in \situations$, a completely vacuous model: $\lupprev(h \vert x_{1:n}) = \max h$ for all local gambles $h \in \setofgambles(\statespace)$ on the next state. 
Then it can be checked easily that $\upprevvovk(f) = \sup f$ for all $f \in \setofextvariables$.
Now let $\statespace \coloneqq \{0,1\}$, and consider the decreasing sequence of events $A_n$, defined by $A_n \coloneqq \{\omega \in \samplespace \colon \omega_i = 1 \text{ for all } 1 \leq i \leq n \} \setminus \{(1,1,1,...)\}$. 
Then $\lim_{n \to +\infty} \indica{A_n} = 0$ pointwise. 
However, as $\upprevvovk(\indica{A_n}) = 1$ for all $n \in \natz$, we have that $\lim_{n \to +\infty} \upprevvovk(\indica{A_n}) = 1$, whereas $\upprevvovk(\lim_{n \to +\infty} \indica{A_n})=\upprevvovk(0) = 0$, so $\upprevvovk$ is not continuous with respect to downward pointwise convergence of gambles.\hfill$\lozenge$
\end{example}
\noindent
This leads us to the conclusion that, in general, $\upprevvovk$ is not continuous with respect to downward---and therefore also pointwise---convergence.
Nevertheless, using a version of L\'evy's zero-one law\iftoggle{arxiv}{}{ \cite{TJoens2018Continuity}}, we can show that $\upprevvovk$ is continuous with respect to upward convergence of extended real variables that are uniformly bounded below, provided that the upper expectation of the limit variable $f$ is finite. 

\begin{theorem}[Upward Convergence Theorem]\label{theorem: upward convergence game}
Consider any non-decreasing sequence of extended real variables $\{f_n\}_{n \in \natz}$ that is uniformly bounded below---i.e. there is an $M \in \reals{}$ such that $f_n \geq M$ for all $n \in \natz{}$---and any extended real variable $f \in \setofextvariables$ such that $\lim_{n \to +\infty} f_n = f$ pointwise. 
If moreover $\upprevvovk(f) < +\infty$, then
\begin{equation*}
\upprevvovk(f) = \lim_{n \to +\infty} \upprevvovk(f_n).
\end{equation*} 
\end{theorem}
\noindent
The initial idea behind the proof is due to Shafer and Vovk, who proved continuity with respect to non-decreasing sequences of indicator gambles~\cite[Theorem 6.6]{Augustin:2014di}.
We have adapted it here to our working with real supermartingales and moreover generalised it to extended real variables.

The following example illustrates the practical relevance of this theorem.

\begin{example}\label{example: upward convergence}
In queuing theory or failure estimation, we are often interested in the time until some event happens and, in particular, in the lower and upper expectation of this time.
As we will illustrate here, Theorem \ref{theorem: upward convergence game} provides a method to approximate such upper expectations. 
The lower expectations can also be approximated, using Theorem~\ref{Lemma Continuity w.r.t. lower cuts} further on; see Example~\ref{example: continuity w.r.t. cuts}.

Consider the simple case where $\statespace{} \coloneqq \{0,1\}$. 
Suppose we are interested in the expected time until the first `$1$' appears. 
In other words, we are interested in the variable $f$ that returns the number of initial successive `$0$'s in a path:
\begin{equation*}
f(\omega) \coloneqq \inf{\big\{k\in\nats\colon \omega_k=1 \big\}}
\text{ for all $\omega \in \samplespace{}$,}
\end{equation*}
where for $\omega = (0,0,0,...)$, $f(\omega) = \inf \emptyset \coloneqq +\infty$.
It is typically infeasible to calculate the upper expectation of this variable directly because it depends on entire paths.
We can remedy this by considering instead, for every $n\in\natz$, the gamble $f_n$, defined by
\begin{equation*}
f_n(\omega) \coloneqq \min{\{f(\omega), n \}}
\text{ for all $\omega\in\Omega$.}
\end{equation*}
For every $n\in\natz$, $f_n$ is clearly $n$-measurable: it only depends on the value of the first $n$ states. Furthermore, $\{f_n\}_{n\in\natz}$ is bounded below by zero, non-decreasing and converges pointwise to $f$. 
Provided that $\upprevvovk(f) < +\infty$, Theorem~\ref{theorem: upward convergence game} therefore implies that $\upprevvovk(f)= \lim_{n \to+\infty} \upprevvovk(f_n)$. 
This allows us to approximate $\upprevvovk(f)$ by $\upprevvovk(f_n)$, for $n$ sufficiently large. 
Since the $n$-measurability of $f_n$ will typically make the computation of $\upprevvovk(f_n)$ feasible, we obtain a practical method for computing $\upprevvovk(f)$. \hfill$\lozenge$
\end{example}
\noindent
As a direct consequence of our Upward Convergence Theorem, we also obtain the following inequality. 

\begin{theorem}[Fatou's Lemma]\label{Fatou}
Consider a sequence of extended real variables $\{f_n\}_{n \in \natz}$ that is uniformly bounded below and let $f \coloneqq \liminf_{n \to +\infty} f_n$. 
If\/ $\upprevvovk(f) < +\infty$, then
\begin{align*}
\upprevvovk(f) \leq \liminf_{n \to +\infty} \upprevvovk(f_n).
\end{align*}
\end{theorem}
\noindent
This result is similar to Fatou's Lemma in measure theory; hence its name. 
It provides an upper bound on the upper expectation of an extended real variable~$f$, in the form of a limit inferior of the upper expectations of any sequence of extended real variables $\{f_n\}_{n \in \natz}$ that is uniformly bounded below and whose limit inferior $\liminf_{n \to +\infty} f_n$ is equal to~$f$. 
Since this last condition is fairly weak, Theorem~\ref{Fatou} has wide applicability.  
In general, the inequality in the statement cannot be reversed, because we do not generally have continuity with respect to pointwise convergence.

\section{Continuity with Respect to Limits of Cuts} 
\label{sec: Continuity with respect to cuts}
Historically, the framework of imprecise probabilities as described by Walley~\cite{Walley:1991vk} has only considered gambles rather than unbounded or even extended real variables.
One important reason for this is that they allow us to use less involved mathematics.
Moreover, when considering unbounded or extended variables in practice, we are typically obliged to work with approximating gambles rather than the original variables.
Considering these arguments, restricting the functional $\upprevvovk$ to gambles would be very tempting, indeed.
However, most practically relevant variables in the context of stochastic processes are in fact unbounded and even extended real-valued; consider for instance hitting or stopping times, as in Examples~\ref{example: upward convergence} and~\ref{example: continuity w.r.t. cuts}.

In Theorem~\ref{theorem: upward convergence game} we have already shown how to approximate upper expectations for bounded below extended variables.
The following theorem allows us to approximate upper expectations for general extended real variables by using sequences of non-increasing lower cuts.

\begin{theorem}
\label{Lemma Continuity w.r.t. lower cuts}
Consider an extended real variable $f \in \setofextvariables$ and, for every $A \in \reals{}$, the variable $f_A$ defined by 
$f_A(\omega) \coloneqq \max\{f(\omega),A\} \text{ for all } \omega \in \samplespace$.
Then  
\begin{align*}
\lim_{A \to -\infty} \upprevvovk(f_A) = \upprevvovk(f).
\end{align*}
\end{theorem}
\noindent
Combining Theorems~\ref{theorem: upward convergence game} and~\ref{Lemma Continuity w.r.t. lower cuts}, we end up with the following result that allows us to move from upper expectations of gambles to upper expectations of general variables.
It also fits within the framework of Troffaes and De Cooman~\cite[Part 2]{troffaes2014}, which provides a general approach to extending coherent lower and upper expectations from gambles to real variables.

\begin{theorem}[Continuity with respect to cuts]\label{theorem Continuity w.r.t. cuts}
Consider any extended real variable $f \in \setofextvariables$ and, for any $A,B \in \reals$ such that $B \geq A$, the gamble $f_{(A,B)}$, defined by
\begin{equation*}\label{eq: cutted variable}
f_{(A,B)}(\omega) \coloneqq 
\begin{cases}
B \text{ if } f(\omega) > B; \\
f(\omega) \text{ if } B \geq f(\omega) \geq A; \\
A \text{ if } f(\omega) < A,
\end{cases}
\text{for all } \omega \in \samplespace.
\end{equation*}
If\/ $\upprevvovk(f) < +\infty$, then
\begin{equation*}
 \lim_{A \to -\infty} \lim_{B \to +\infty} \upprevvovk(f_{(A,B)}) = \upprevvovk(f).
\end{equation*}
\end{theorem}

\begin{example}\label{example: continuity w.r.t. cuts}
Consider the same state space $\statespace{}$ and the same variables $f$ and $f_n$ as in Example \ref{example: upward convergence}.
We have already shown there how to approximate the upper expectation $\upprevvovk(f)$ of $f$ by $\upprevvovk(f_n)$.
Now, we want to approximate the lower expectation $\lowprevvovk$---defined by $\lowprevvovk(g)\coloneqq-\upprevvovk(-g)$ for all $g\in\setofextvariables$---of $f$.
As $\{ f_n \}_{n \in \nats{}}$ is an increasing sequence of upper cuts of $f$, $\{-  f_n \}_{n \in \nats{}}$ is a decreasing sequence of lower cuts of $-f$.
Hence, it follows from Theorem~\ref{Lemma Continuity w.r.t. lower cuts} that $\lim_{n \to +\infty} \upprevvovk{}(- f_n) = \upprevvovk{}(- f)$, and therefore, using conjugacy, $\lim_{n \to +\infty} - \lowprevvovk{}( f_n) = - \lowprevvovk{}(f)$, or, equivalently, $\lim_{n \to +\infty} \lowprevvovk{}( f_n) = \lowprevvovk{}(f)$.
Hence, in the same way as was described in Example~\ref{example: upward convergence}, we now also have a constructive method for approximating the lower expectation of the variable $f$. \hfill$\lozenge$
\end{example}

\section{Conclusion}
\label{sec: Conclusion}

Among the continuity properties derived in this paper, the continuity with respect to cuts is the more remarkable, as it allows us to limit ourselves, for the larger part, to the study of $\upprevvovk$ on gambles rather than the study of $\upprevvovk$ on extended real variables.
Although the functional $\upprevvovk$ is not continuous with respect to general downward convergence, it is thus continuous for a particular way of downward convergence: sequences of non-increasing lower cuts.
These results hold provided that the upper expectation $\upprevvovk(f)$ of the limit variable $f$ is finite. 
The case where $ \upprevvovk(f)  = + \infty$ is largely left unexplored.

There is also an interesting connection between the game-theoretic functional $\upprevvovk$ and the measure-theoretic Lebesgue integral when all local models are assumed precise.
It was already pointed out by Shafer and Vovk~\cite[Chapter 8]{Shafer:2005wx} that for indicator gambles $\indica{A}$ of events $A$ in the $\sigma$-algebra created by the cylinder events, $\upprevvovk(\indica{A})$ is equal to the Lebesgue integral of $\indica{A}$ when the global measure is defined according to the Ionescu-Tulcea Theorem~\cite[p.249]{shiryaev1995probability}.
Using our results, we aim to generalise the connection between both operators and study the extend to which they are equal. 
We leave this as future work.

Another topic of further research is the continuity of $\upprevvovk$ with respect to the pointwise convergence of $n$-measurable gambles.
We suspect that, in order to establish the behaviour of $\upprevvovk$ with respect to this particular type of convergence, it will pay to investigate the potentially strong link with the concept of natural extension~\cite{Walley:1991vk} and, as a consequence, the special status of $\upprevvovk$ with respect to other extending functionals.


\iftoggle{arxiv}{
\appendix
\section{Additional Preliminaries}

\subsection{Upper Expectations}\label{section: lowprev}
Consider any non-empty set $\mathscr{Y}$.
We have already defined a coherent upper expectation $\upprev{}$ as a functional on $\setofgambles{}(\mathscr{Y})$ that satisfies the coherence axioms \ref{coherence 1}--\ref{coherence 3}, and have introduced its conjugate lower expectation $\lowprev{}$. 
We will moreover need the following properties of these operators, which follow at once from~\cite[2.6.1]{Walley:1991vk}. 
For all gambles $f, g \in \setofgambles(\mathscr{Y})$, any sequence $\{f_n\}_{n \in \natz}$ of gambles in $\setofgambles{}(\mathscr{Y})$ and all real $\mu$,
\begin{enumerate}[ref={\upshape E}\arabic*,label={\upshape E}\arabic*.,resume]
\item \label{coherence 4}
$f \leq g \Rightarrow \upprev(f) \leq \upprev(g)$ and $\lowprev(f) \leq \lowprev(g)$;
\item \label{coherence 5}
$\inf f \leq \lowprev(f) \leq \upprev(f) \leq \sup f$;
\item \label{coherence 6}
$\upprev(f + \mu) = \upprev(f) + \mu$ and $\lowprev(f + \mu) = \lowprev(f) + \mu$;
\item \label{coherence 7}
$\lim_{n \to +\infty} \sup{\vert f - f_n \vert} = 0 \Rightarrow \lim_{n \to +\infty} \vert \upprev(f) - \upprev(f_n) \vert = 0$.
\end{enumerate}
It follows from property \ref{coherence 7} that any coherent upper expectation $\upprev{}$ is continuous with respect to uniform convergence.

As the restriction of $\upprevvovk(\cdot\vert s)$ to $\setofgambles{}(\samplespace{})$ is a coherent upper expectation on $\setofgambles{}(\samplespace{})$, it follows at once that $\upprevvovk(\cdot\vert s)$ also satisfies \ref{coherence 4}--\ref{coherence 7} on $\setofgambles{}(\samplespace{})$.
Moreover, by~\cite{DECOOMAN201618}, some of them continue to hold when we consider $\upprevvovk(\cdot\vert s)$ on the extended real variables $\setofextvariables{}$. 
In particular, \ref{coherence 4}--\ref{coherence 6} remain unchanged, and will be denoted by respectively \ref{coherence 4}*--\ref{coherence 6}* when we consider extended real variables.

\subsection{Cuts and processes}

We have already defined situations, paths, variables, processes.
We say that a path $\omega \in \samplespace{}$ \emph{goes through} a situation $s \in \situations{}$ when there is some $n \in \natz{}$ such that $\omega^n = s$.
We write that $s \sqsubseteq t$, and say that $s$ \emph{precedes} $t$ or that $t$ \emph{follows} $s$, when every path that goes through $t$ also goes through $s$. 
When $s \sqsubseteq t$ and $ s \not= t$, we write $s \sqsubset t$. 
When neither $s \sqsubseteq t$ nor $t \sqsubseteq s$, we say that $s$ and $t$ are \emph{incomparable}.

In our proofs, we will also need the concept of a cut.
A cut $U$ is collection of pairwise incomparable situations. 
For any two cuts $U$ and $V$, we can define the following sets of situations:
\begin{align*}
[U,V] \coloneqq \{s \in \situations : (\exists u \in U)(\exists v \in V)u \sqsubseteq s \sqsubseteq v \}, \\
[U,V) \coloneqq \{s \in \situations : (\exists u \in U)(\exists v \in V)u \sqsubseteq s \sqsubset v \}, \\
(U,V] \coloneqq \{s \in \situations : (\exists u \in U)(\exists v \in V)u \sqsubset s \sqsubseteq v \}, \\
(U,V) \coloneqq \{s \in \situations : (\exists u \in U)(\exists v \in V)u \sqsubset s \sqsubset v \}.
\end{align*}
We call a cut $U$ \emph{complete} if for all $\omega \in \Omega$ there is some $u \in U$ such that $\omega \in \Gamma(u)$. 
Otherwise, we call $U$ \emph{partial}. 
We will also use the simpler notation $s$ to denote the cut $\{s\}$ that consists of the single situation $s \in \situations$. 
In this way we can define $[U,s]$, $[U,s)$, $[s,V]$, ... in a similar way as above.
We also write $U \sqsubset V$ if $(\forall v \in V)(\exists u \in U) \ u \sqsubset v$.
Analogously as before, we say that a path $\omega \in \samplespace{}$ goes through a cut $U$ when there is some $n \in \natz{}$ such that $\omega^n \in U$.


Recall that a gamble process $\mathscr{D}$ is a process that associates with any situation~$s$ a gamble $\mathscr{D}(s) \in \setofgambles(\statespace)$ on $X_{n+1}$ and that, with any real process $\mathscr{L}$, we can associate a gamble process $\Delta \mathscr{L}$, called its process difference.
Conversely, with a gamble process $\mathscr{D}$, we can associate a real process $\Delta^{-1}\mathscr{D}$, called its \emph{summation process}, defined by
\begin{equation*}
\Delta^{-1}\mathscr{D}(x_{1:n}) \coloneqq \sum_{k=0}^{n-1} \mathscr{D}(x_{1:k})(x_{k+1}) \text{ for all } n \in \nats \text{ and } x_{1:n} \in \statespaceseq{1}{n},
\end{equation*}
where we also let $\Delta^{-1}\mathscr{D}(\Box) \coloneqq 0$.

With any real process $\mathscr{L}$ such that $\mathscr{L}(s) \not= 0$ for all $s \in \situations$, we can associate a gamble process $\mu \mathscr{L}$, called its \emph{process multiplier}, as follows. For every situation $x_{1:n}$ the gamble $\mu \mathscr{L}(x_{1:n}) \in \setofgambles(\statespace)$ is defined by
\begin{equation*}
\mu \mathscr{L}(x_{1:n})(x_{n+1}) \coloneqq \frac{\mathscr{L}(x_{1:n+1})}{\mathscr{L}(x_{1:n})} \text{ for all } x_{n+1} \in \statespace.
\end{equation*}
Conversely, with a gamble process $\mathscr{D}$, we can associate a real process $\mu^{-1}\mathscr{D}$, called its \emph{product process}, defined by
\begin{equation*}
\mu^{-1}\mathscr{D}(x_{1:n}) \coloneqq \prod_{k=0}^{n-1} \mathscr{D}(x_{1:k})(x_{k+1}) \text{ for all } n \in \nats \text{ and } x_{1:n} \in \statespaceseq{1}{n},
\end{equation*}
where we also let $\mu^{-1}\mathscr{D}(\Box) \coloneqq 1$.

We will also say that a supermartingale $\martingale{}$ is a \emph{test supermartingale} if it is non-negative and $\martingale(\Box) = 1$.

\section{Proofs of the results in Section \ref{sec: Continuity properties}}\label{section: Proofs Continuity}

\subsection{L\'evy's zero-one law}\label{subsec: Levy}

One of the key results needed for the proofs in this paper is a particular version of L\'evy's zero-one law suitable in our context of imprecise probability trees.
Takemura, Shafer and Vovk have already proved a version of this law in Reference~\cite{shafer2011levy}.
However, the use of real supermartingales in our Definition~\ref{def:upperexpectation} requires a slightly different approach.
The idea remains essentially the same, the technical details differ.

\begin{lemma}[\protect{\cite[Lemma 1] {DECOOMAN201618}}] \label{lemma1}
	Consider any supermartingale $\martingale \in \setofsupmart$ and any situation~$s \in \situations$. Then
	\begin{equation*}
		\martingale(s) \geq \inf_{\omega \in \Gamma(s)} \limsup \martingale(\omega) \geq \inf_{\omega \in \Gamma(s)} \liminf \martingale(\omega).
	\end{equation*}
\end{lemma}

\begin{lemma}\label{LemmaBoundedSupermartingale}
	Consider any bounded below supermartingale $\martingale \in \setofsupmartb$ and, for any $B \in \reals$, the real process $\martingale_B$ defined by
	\begin{equation*}
	\martingale_B(s) \coloneqq \min\{\martingale(s), B\} \text{ for all } s \in \situations.
	\end{equation*} 
	Then $\martingale_B$ is also a bounded below supermartingale.
\end{lemma}
\begin{proof}
	It is clear that, since $\martingale$ is a bounded below real process, $\martingale_B$ is also a bounded below real process. 
	Now fix any $x_{1:n} \in \situations$.
	We consider two cases. 
	If $\martingale(x_{1:n}) \leq B$, then
	\begin{multline*}
	\Delta \martingale_B(x_{1:n})(x_{n+1}) 
	= \min\{\martingale(x_{1:n+1}), B\} - \martingale(x_{1:n}) \\
	\leq \martingale(x_{1:n+1}) - \martingale(x_{1:n})
	= \Delta \martingale (x_{1:n})(x_{n+1}) \text{ for all $x_{n+1} \in \statespace$.}
	\end{multline*}
	It therefore follows that $\Delta \martingale_B(x_{1:n}) \leq \Delta \martingale(x_{1:n})$, and hence, by the monotonicity \ref{coherence 4} of $\lupprev{}(\cdot \vert x_{1:n})$, also that
	\begin{align*}
	\lupprev(\Delta \martingale_B(x_{1:n}) \vert x_{1:n}) \leq \lupprev(\Delta \martingale(x_{1:n}) \vert x_{1:n}) \leq 0.
	\end{align*}
	\noindent
	If $\martingale(x_{1:n}) > B$, then
	\begin{align*}
	\Delta \martingale_B(x_{1:n})(x_{n+1}) = \min\{\martingale(x_{1:n+1}), B\} - B \leq 0 \text{ for all } x_{n+1} \in \statespace.
	\end{align*}
	Hence, $\Delta \martingale_B(x_{1:n}) \leq 0$ and therefore, by \ref{coherence 1}, also $\lupprev(\Delta \martingale_B(x_{1:n}) \vert x_{1:n}) \leq 0$.

	We conclude that $\lupprev(\Delta \martingale_B(s) \vert s) \leq 0$ for all $s \in \situations$, so $\martingale_B$ is indeed a bounded below supermartingale.
	\qed
\end{proof}

\begin{lemma}\label{Lemma: the minimum - limit inferior switch}
	Consider any real process $\process{}$ and any path $\omega \in \samplespace{}$. Then 
	\begin{equation*}
	 \min \{B,\liminf_{n \to +\infty} \process(\omega^n)\} = \liminf_{n \to +\infty} \min\{B,\process(\omega^n)\} \text{ for all $B \in \reals$.}
	\end{equation*} 
\end{lemma}
\begin{proof}
	Consider any $B \in \reals$.
	It is easy to check that 
	\begin{equation*}
	\min \{B,\liminf_{n \to +\infty} \process(\omega^n)\} \geq \liminf_{n \to +\infty} \min\{B,\process(\omega^n)\}.
	\end{equation*} 
	We prove the converse inequality by contradiction. 
	Suppose that 
	\begin{align}
	\min \{B,\liminf_{n \to +\infty} \process(\omega^n)\} &> \liminf_{n \to +\infty} \min\{B,\process(\omega^n)\} \nonumber,
	\end{align}
	or, equivalently, that
	\begin{equation*}
	\min \{B,\liminf_{n \to +\infty} \process(\omega^n)\} > \sup_{m} \inf_{n \geq m} \min\{B,\process(\omega^n)\}.
	\end{equation*}
	Then there is some $\epsilon > 0$ such that 
	\begin{equation*}
	\min \{B,\liminf_{n \to +\infty} \process(\omega^n)\} - \epsilon > \inf_{n \geq m} \min\{B,\process(\omega^n)\} = \min\{B, \inf_{n \geq m} \process(\omega^n)\}
	\end{equation*}
	for all $ m \in \natz{}$.
	Since $\min \{B,\liminf_{n \to +\infty} \process(\omega^n)\} - \epsilon \leq B$, this implies that
	\begin{align*}
	\inf_{n \geq m} \process(\omega^n) < \min \{B,\liminf_{n \to +\infty} \process(\omega^n)\} - \epsilon \leq \liminf_{n \to +\infty} \process(\omega^n) - \epsilon \text{ for all $m \in \natz{}$,}
	\end{align*}
	from which we infer that
	\begin{align*}
		 \inf_{n \geq m} \process(\omega^n) < \sup_{k} \inf_{n \geq k} \process(\omega^n) - \epsilon \text{ for all $m \in \natz{}$,}
	\end{align*} 
	contradicting the definition of the supremum operator.
	\qed

\end{proof}

\begin{lemma}\label{lemma: multiplicative supermartingale}
A positive real process $\martingale{}$ is a supermartingale if and only if, for all $s \in\situations{}$, $\lupprev{}(\mu\martingale{}(s) \vert s) \leq 1$.
\end{lemma}
\begin{proof}
Fix any $s \in\situations{}$.
Since $\martingale{}(s) \not= 0$, it is clear that $\martingale{}(s) \, \mu\martingale{}(s) = \Delta\martingale{}(s) + \martingale{}(s)$.
Since moreover $\martingale{}(s) > 0$, we have, by the non-negative homogeneity [\ref{coherence 3}] of $\lupprev{}(\cdot \vert s)$, that $\lupprev{}(\martingale{}(s) \, \mu\martingale{}(s) \vert s) = \martingale{}(s) \lupprev{}(\mu\martingale{}(s) \vert s)$ and furthermore, by the constant additivity [\ref{coherence 6}] of $\lupprev{}(\cdot \vert s)$, that $\lupprev{}(\Delta\martingale{}(s) + \martingale{}(s) \vert s) = \lupprev{}(\Delta\martingale{}(s) \vert s) + \martingale{}(s)$.
Again using the positivity of $\martingale{}(s)$, it is now obvious that $\lupprev{}(\Delta\martingale{}(s) \vert s) \leq 0$ if and only if $\lupprev{}(\mu\martingale{}(s) \vert s) \leq 1$.
\qed
\end{proof}

\begin{lemma}\label{lemma: Convex sum}
Consider any $\alpha \in \posreals{}$ and any countable collection $\{\martingale{}_n\}_{n \in \nats{}}$ of positive test supermartingales such that $\mu \martingale_n{}(s) \leq \alpha$ for all $s \in \situations{}$.
Let $\martingale{}$ be any convex combination of them:
\begin{equation*}
	\martingale{}(s) \coloneqq \sum_{n \in \nats{}} \lambda_n \martingale{}_n(s) \text{ for all } s \in \situations{},
\end{equation*}
where the coefficients $\lambda_n \geq 0$ sum to $1$.
Then $\martingale{}$ is also a positive test supermartingale such that $\mu \martingale{}(s) \leq \alpha$ for all $s \in \situations{}$.
\end{lemma}
\begin{proof}
	Since all the elements in the sum are non-negative and at least one positive, $\martingale{}(s)$ is a, possibly extended, positive real number for all $s \in \situations{}$.
	It is also clear that $\martingale{}(\Box) = 1$.

	We prove by induction that $\martingale{}$ is a real process such that $\mu \martingale{}(s) \leq \alpha$ for all $s \in \situations{}$.
	We already know that $\martingale{}(\Box)$ is a positive real number.
	Now, consider any $x_{1:k} \in \situations{} \setminus \Box$ and suppose that $\martingale{}(x_{1:k-1})$ is a positive real number.
	Then, by the definition of $\martingale{}$, we have
	\begin{align*}
		\martingale{}(x_{1:k}) 
		= \sum_{n \in \nats{}} \lambda_n \martingale{}_n(x_{1:k}) 
		&= \sum_{n \in \nats{}} \lambda_n \martingale{}_n(x_{1:k-1}) \mu\martingale{}_n(x_{1:k-1})(x_k)  \\
		&\leq \sum_{n \in \nats{}} \lambda_n \martingale{}_n(x_{1:k-1}) \, \alpha \\
		&= \alpha \sum_{n \in \nats{}} \lambda_n \martingale{}_n(x_{1:k-1}) \\
		&= \alpha \martingale{}(x_{1:k-1}),
	\end{align*}
	where the inequality follows from $\mu \martingale_n{}(x_{1:k-1})(x_k) \leq \alpha$ and the positivity of $\martingale{}_n$ and the non-negativity of $\lambda_n$.
	So it is clear that $\martingale{}(x_{1:k})$ is dominated by a positive real number, and since $\martingale{}$ is moreover positive, it follows that $\martingale{}(x_{1:k})$ is a positive real number.
	Hence, by induction, $\martingale{}$ is a positive real process.
	It moreover follows from the reasoning above that $\mu \martingale{}(x_{1:k-1})(x_k) \leq \alpha$ for any $x_{1:k} \in \situations{} \setminus \Box$, or equivalently, that $\mu \martingale{}(s) \leq \alpha$ for any $s \in \situations{}$.

	We now prove that $\martingale{}$ is a supermartingale, and therefore, since $\martingale{}(\Box) = 1$, a test supermartingale.
	Fix any situation $s \in \situations{}$.
	Recall that the value of the gamble $\mu\martingale{}(s)$ in some state $x \in \statespace{}$ is given by
	\begin{align}\label{eq: proof lemma convex sum}
		\mu\martingale{}(s)(x) 
		 = \martingale{}(s)^{-1} \martingale{}(sx)
		 &= \martingale{}(s)^{-1} \sum_{n \in \nats{}} \lambda_n \martingale{}_n(sx) \nonumber \\
		&= \martingale{}(s)^{-1} \sum_{n \in \nats{}} \lambda_n \martingale{}_n(s) \mu\martingale{}_n(s)(x).
	\end{align}
	Let the sequence $\{f_n\}_{n \in \nats{}}$ of local gambles on $\statespace{}$ be defined by
	\begin{align*}
		f_n(x) \coloneqq \sum_{i=1}^n \lambda_i \martingale{}_i(s) \mu\martingale{}_i(s)(x) \text{ for all } x \in \statespace{} \text{ and all } n \in \nats{}. 
	\end{align*}
	If we fix any state $x \in \statespace{}$, we infer from \eqref{eq: proof lemma convex sum} that 
	\begin{align*}
		 \mu\martingale{}(s)(x) =  \martingale{}(s)^{-1} \lim_{n \to +\infty} f_n(x).
	\end{align*}
	Since $\{f_n(x)\}_{n \in \nats{}}$ is clearly a non-decreasing sequence, its limit exists. 
	Moreover, that $\martingale{}$ is a real process implies that the sequence $\{f_n(x)\}_{n \in \nats{}}$ converges to the real number $\lim_{n \to +\infty} f_n(x)$.
	Now, as the argument above holds for any $x \in \statespace{}$, and since we work in a local state space $\statespace$ that is finite, it follows that the sequence $\{f_n\}_{n \in \nats{}}$ of local gambles converges \emph{uniformly} to the local gamble $\lim_{n \to +\infty} f_n$.
	Hence, because the coherent upper prevision $\lupprev{}(\cdot \vert s)$ is continuous with respect to uniform convergence [\ref{coherence 7}], and using the non-negative homogeneity~[\ref{coherence 3}] and the subadditivity~[\ref{coherence 2}] of $\lupprev{}(\cdot \vert s)$, we find that 
	\begin{align*}
	\lupprev( \mu\martingale{}(s) \vert s) 
	&= \lupprev \big( \martingale{}(s)^{-1} \lim_{n \to +\infty} f_n  \big\vert s\big) \\
	&\overset{\text{\ref{coherence 3}}}{=}\martingale{}(s)^{-1} \lupprev \big( \lim_{n \to +\infty} f_n  \big\vert s\big) \\
	&\overset{\text{\ref{coherence 7}}}{=} \martingale{}(s)^{-1} \lim_{n \to +\infty} \lupprev \big( f_n  \big\vert s\big) \\
	&\overset{\text{\ref{coherence 2}, \ref{coherence 3}}}{\leq} \martingale{}(s)^{-1} \lim_{n \to +\infty} \sum_{i=1}^n \lambda_i \martingale{}_i(s) \lupprev{}(\mu\martingale{}_i(s) \vert s) \\
	&\leq \martingale{}(s)^{-1} \sum_{n \in \nats{}} \lambda_n \martingale{}_n(s) = 1,
	\end{align*}
	where the last inequality holds because, by Lemma~\ref{lemma: multiplicative supermartingale}, $\lupprev (\mu \martingale{}_i(s) \vert s)\leq 1$ for all $i \in \nats{}$.
	Now, since $\lupprev( \mu\martingale{}(s) \vert s) \leq 1$ for any situation $s \in \situations{}$, Lemma~\ref{lemma: multiplicative supermartingale} implies that $\martingale{}$ is indeed a supermartingale, and therefore a positive test supermartingale. 
	\qed
\end{proof}

\noindent
Consider any two variables $g,h \in \setofextvariables{}$ and any situation~$s \in \situations{}$.
From now on, we use $g \leq_s f$ to denote that $g(\omega) \leq f(\omega)$ for all $\omega \in \Gamma(s)$, and similarly for $\geq_s$, $>_s$ and $<_s$.

We say that an event $A \subseteq \samplespace{}$ is \emph{strictly almost sure (s.a.s.)} if there is a test supermartingale that converges to $+\infty$ on $\samplespace{} \setminus A$. 

\begin{theorem}[L\'evy's zero-one law]\label{Bounded Levy}
	Consider any gamble $f \in \setofgambles(\samplespace{})$ and any real number $\alpha > 1$.
	Then the event
	\begin{align*}
		A \coloneqq \Big\{\omega \in \samplespace{} \colon \liminf_{n \to +\infty} \upprevvovk(f \vert \omega^n) \geq f(\omega) \Big\}
	\end{align*}
	is strictly almost sure.
	Moreover, the test supermartingale $\mathscr{T}$ that converges to $+\infty$ on $\samplespace{} \setminus A$ can be chosen such that it is positive and that $\mu\mathscr{T}(s) \leq \alpha$ for all $s \in \situations{}$.
\end{theorem}
\begin{proof}
	 Since $\upprevvovk(\cdot \vert s)$ is constant additive [\ref{coherence 6}*], we have, for any $\beta \in \reals{}$, that $\liminf_{n \to +\infty} \upprevvovk(f \vert \omega^n) \geq f(\omega)$ if and only if $\liminf_{n \to +\infty} \upprevvovk(f + \beta \vert \omega^n) \geq f(\omega) + \beta$.
	 Therefore, and because f is bounded, we can assume without loss of generality that $f$ is a gamble such that both $\inf f > 0$ and $\nicefrac{\sup f}{\inf f} \leq \alpha$.
	 It suffices to add a sufficiently large positive real number to the gamble initially considered. 

	We now associate with any couple of rational numbers $0<a<b$ the following recursively constructed sequences of cuts $\{U_k^{a,b}\}_{k \in \natz{}}$ and $\{V_k^{a,b}\}_{k \in \nats{}}$. 
	Let $U_0^{a,b} \coloneqq \{ \Box \}$ and, for $k \in \nats$,
	\begin{enumerate}
		\item 
		let
		\begin{align*}
			V_k^{a,b} \coloneqq \{ s \in \situations{} \colon  U_{k-1}^{a,b} \sqsubset s , \ \upprevvovk(f \vert s) < a \text{ and } (\forall t \in (U_{k-1}^{a,b} , s)) \ \upprevvovk(f \vert t) \geq a \};
		\end{align*} 
		\item if $V_k^{a,b}$ is non-empty, choose a positive supermartingale $\martingale_{k}^{a,b} \in \setofsupmart$ such that 
		\begin{align*}
		&\inf f \leq \martingale_k^{a,b}(t) \leq \sup f \text{ for all } t \in \situations{}, \\
		&\martingale_{k}^{a,b}(s) < a \text{ and } \liminf \martingale_k^{a,b}  \geq_s f \text{ for all } s \in V_k^{a,b}
		\end{align*}
		and let
		\begin{align*}
			U_k^{a,b} \coloneqq \{s \in \situations{} \colon V_{k}^{a,b} \sqsubset s : \martingale_{k}^{a,b}(s) > b \text{ and } (\forall t \in (V_{k}^{a,b} , s)) \ \martingale_{k}^{a,b}(t) \leq b \};
		\end{align*}
		if $V_k^{a,b}$ is empty, let $U_k^{a,b} \coloneqq \emptyset$.
	\end{enumerate}
	The cuts $U_k^{a,b}$ and $V_k^{a,b}$ can be partial or complete.
	We now first show that, if $V_k^{a,b}$ is non-empty, there always is a supermartingale $\martingale_{k}^{a,b}$ that satisfies the conditions above. 
	We infer from the definition of the cut $V_k^{a,b}$ that
	\begin{equation*}
	 \inf \bigg\{ \martingale(s) :  \martingale \in \setofsupmartb \text{ and } \liminf \martingale \geq_s f \bigg\} < a \text{ for all } s \in V_k^{a,b}.
	 \end{equation*}
	So, for all $s \in V_k^{a,b}$, we can choose a supermartingale $\martingale_{k,s}^{a,b}$ such that $\martingale_{k,s}^{a,b}(s) < a$ and $\liminf \martingale_{k,s}^{a,b} \geq_s f$. 
	Consider now the real process $\martingale_{k,\ast}^{a,b}$ defined, for all $ t \in \situations{}$, by
	\begin{align*}
		\martingale_{k,\ast}^{a,b}(t) \coloneqq
		\begin{cases}
		\martingale_{k,s}^{a,b}(t) \ &\text{ if } s \sqsubseteq t \text{ for some } s \in V_k^{a,b}; \\
		a \ &\text{ otherwise.}
		\end{cases}
	\end{align*}
	It is clear that $\martingale_{k,\ast}^{a,b}(s) < a \text{ and } \liminf \martingale_{k,\ast}^{a,b}  \geq_s f$ for all $s \in V_k^{a,b}$.

	We show that $\martingale_{k,\ast}^{a,b}$ is also a supermartingale.
	Fix any $t \in \situations{}$ and consider two cases. 
	If $V_k^{a,b} \sqsubseteq t$, then $\Delta\martingale_{k,\ast}^{a,b}(t) = \Delta\martingale_{k,s}^{a,b}(t)$ for some $s \in V_k^{a,b}$, and therefore $\lupprev{}(\Delta\martingale_{k,\ast}^{a,b}(t) \vert t) = \lupprev{}(\Delta\martingale_{k,s}^{a,b}(t) \vert t) \leq 0$.
	If $V_k^{a,b} \not\sqsubseteq t$, then, for any $x \in \statespace$, we have either $tx \in V_k^{a,b}$ and therefore $\Delta\martingale_{k,\ast}^{a,b}(t)(x) = \martingale_{k,tx}^{a,b}(tx) - a < 0$, or $tx \not\in V_k^{a,b}$ and therefore $\Delta\martingale_{k,\ast}^{a,b}(t)(x) = a - a = 0$.
	Hence, we have $\Delta\martingale_{k,\ast}^{a,b}(t) \leq 0$, and therefore, by \ref{coherence 1}, $\lupprev{}(\Delta\martingale_{k,\ast}^{a,b}(t) \vert t) \leq 0$. 
	As a consequence, $\martingale_{k,\ast}^{a,b}$ is a supermartingale. 

	Furthermore, note that $a > \inf f$.
	Indeed, for any $s \in V_k^{a,b}$, it follows directly from Lemma \ref{lemma1} that
	\begin{align*}
	a > \martingale_{k,s}^{a,b}(s) \geq \inf_{\omega \in \Gamma(s)} \liminf \martingale_{k,s}^{a,b}(\omega) \geq \inf_{\omega \in \Gamma(s)} f \geq \inf f.
	\end{align*}
	Since $V_k^{a,b}$ is non-empty, this implies that $a > \inf f$.
	Therefore, we have that $\martingale_{k,\ast}^{a,b}(t) = a > \inf f$ for any $t \in \situations{}$ such that $V_k^{a,b} \not\sqsubseteq t$.
	On the other hand, for any $t \in \situations{}$ such that $V_k^{a,b} \sqsubseteq t$, it follows from Lemma \ref{lemma1} that $\martingale_{k,\ast}^{a,b}(t) = \martingale_{k,s}^{a,b}(t) \geq \inf f$ for some $s \in V_k^{a,b}$.
	Hence, $\martingale_{k,\ast}^{a,b}(t) \geq \inf f > 0$ for all $t \in \situations{}$, which implies that $\martingale_{k,\ast}^{a,b}$ is bounded below.

	Now, let $\martingale_k^{a,b}$ be defined by $\martingale_k^{a,b}(t) \coloneqq \min\{ \martingale_{k,\ast}^{a,b}(t), \ \sup f \}$ for all $t \in \situations{}$.
	Then it is clear that $0 < \inf f \leq \martingale_k^{a,b}(t) \leq \sup f$ for all $t \in \situations{}$ and moreover, by Lemma \ref{LemmaBoundedSupermartingale}, $\martingale_k^{a,b}$ is a bounded below supermartingale that is clearly positive.
	Furthermore, it follows from Lemma \ref{Lemma: the minimum - limit inferior switch} that, for any $\omega \in \samplespace{}$,
	\begin{align*}
	\liminf_{n \to +\infty} \martingale_k^{a,b}(\omega^n) = \min \Big\{ \liminf_{n \to +\infty} \martingale_{k,\ast}^{a,b}(\omega^n), \ \sup f \Big\}.
	\end{align*}
	Since, for all $s \in V_k^{a,b}$ and $\omega \in \Gamma(s)$, $\liminf_{n \to +\infty} \martingale_{k,\ast}^{a,b}(\omega^n) \geq f(\omega)$ and $\sup f \geq f(\omega)$, this implies that $\liminf \martingale_k^{a,b} \geq_s f$ for all $s \in V_k^{a,b}$.
	Hence, $\martingale_k^{a,b}$ is a supermartingale satisfying the conditions above.

	Since all $\martingale_k^{a,b}$ are positive, we can use their process multipliers $\mu \martingale_k^{a,b}$, which are positive as well, to construct a new gamble process $\mu \mathscr{T}^{a,b}$, defined by
	 \begin{align*}
	 	\mu \mathscr{T}^{a,b} (s) \coloneqq \begin{cases}
	 		\mu \martingale_{k}^{a,b}(s)   &\text{ if } s \in  [V_{k}^{a,b} , U_{k}^{a,b}) \text{ for some } k \in \nats;\\
	 		1 &\text{ otherwise,}
	 		\end{cases}
	 		\text{ for all } s \in \situations{}.
	 \end{align*}
	 We next prove that the corresponding real process $\mathscr{T}^{a,b} \coloneqq \mu^{-1} (\mu \mathscr{T}^{a,b})$ is a positive test supermartingale that converges to $+ \infty$ on all paths $\omega \in \samplespace{}$ such that
	 \begin{equation}\label{Eq: levy conv to infty}
	 \liminf_{n \to +\infty} \upprevvovk(f \vert \omega^n) < a < b < f(\omega),
	 \end{equation}
	 and that moreover $\mu\mathscr{T}^{a,b}(s) \leq \alpha$ for all $s \in \situations{}$.
	 It follows from the definition of $\mathscr{T}^{a,b}$ that $\mathscr{T}^{a,b} (\Box) \coloneqq 1$. 
	 For all $k \in \nats{}$ such that $V_k^{a,b}$ is non-empty, we have that $0 < \inf f \leq \martingale_k^{a,b}(t) \leq \sup f$ for all $t \in \situations{}$, and therefore $\mu \martingale_k^{a,b}$ is a positive gamble process such that $\mu \martingale_k^{a,b}(s) \leq \nicefrac{\sup f}{\inf f} \leq \alpha$ for all $s \in \situations{}$.
	 This implies directly that $\mu\mathscr{T}^{a,b}$ is also a positive gamble process and, together with $\alpha > 1$, it also implies that $\mu\mathscr{T}^{a,b}(s) \leq \alpha$ for all $s \in \situations{}$.
	 That $\mu\mathscr{T}^{a,b}$ is a positive gamble process implies, together with $\mathscr{T}^{a,b} (\Box) = 1$, that $\mathscr{T}^{a,b}$ is a real positive process.

	 Furthermore, for any $s \in \situations$, either $\mu \mathscr{T}^{a,b}(s) = \mu \martingale_{k}^{a,b}(s)$ for some $k \in \nats{}$, and therefore Lemma~\ref{lemma: multiplicative supermartingale} implies that $\lupprev(\mu \mathscr{T}^{a,b}(s) \vert s) = \lupprev(\mu \martingale_{k}^{a,b}(s) \vert s) \leq 1$ because $\martingale_{k}^{a,b}$ is a supermartingale, either $\mu \mathscr{T}^{a,b}(s) = 1$, which implies, together with \ref{coherence 5}, that $\lupprev(\mu \mathscr{T}^{a,b}(s) \vert s) = 1$.
	 As a result, we have that $\lupprev(\mu \mathscr{T}^{a,b}(s) \vert s) \leq 1$ for all $s \in \situations{}$, and hence, we infer from Lemma~\ref{lemma: multiplicative supermartingale} that $\mathscr{T}^{a,b}$ is a positive test supermartingale.

	 Next, we show that $\mathscr{T}^{a,b}$ converges to $+ \infty$ on all paths $\omega \in \samplespace{}$ for which (\ref{Eq: levy conv to infty}) holds.
	 Consider such a path $\omega$.
	 Then $\omega$ goes through all the cuts $U_0^{a,b} \sqsubset V_1^{a,b} \sqsubset U_1^{a,b} \sqsubset ... \sqsubset V_n^{a,b} \sqsubset U_n^{a,b} \sqsubset ... $. 
	 Indeed, it is trivial that $\omega$ goes through $U_0^{a,b} = \{\Box\}$. 
	 Furthermore, it follows from $\liminf_{n \to +\infty} \upprevvovk(f \vert \omega^n) < a$ that there exists, for all $m \in \natz{}$, some $n \in \nats$ such that $n > m$ and $\upprevvovk(f \vert \omega^n) < a$. 
	 Take the first $n_1 \in \nats$ such that $\upprevvovk(f \vert \omega^{n_1}) < a$. 
	 Then it follows from the definition of $V_1^{a,b}$ that $\omega^{n_1} \in V_1^{a,b}$. 
	 Next, it follows from $\liminf_{n \to +\infty} \martingale_1^{a,b} (\omega^n) \geq f(\omega) > b$ that there exists some $m_1 \in \nats$ for which $m_1 > n_1$ and $\martingale_1^{a,b} (\omega^{m_1}) > b$. 
	 Take the first such $m_1$, then it follows from the definition of $U_1^{a,b}$ that $\omega^{m_1} \in U_1^{a,b}$. 
	 Repeating similar arguments over and over again allows us to conclude that $\omega$ indeed goes through all the cuts $U_0^{a,b} \sqsubset V_1^{a,b} \sqsubset U_1^{a,b} \sqsubset ... \sqsubset V_n^{a,b} \sqsubset U_n^{a,b} \sqsubset ... $.

	 In what follows, we use the following notation. 
	 For any situation~$s$ and for any $k \in \natz{}$, when $U_k^{a,b} \sqsubset s$, we denote by $u_k^s$ the (necessarily unique) situation in $U_k^{a,b}$ such that $u_k^s \sqsubset s$; observe that $u_0^s = \Box$. 
	 Similarly, for any $k \in \nats{}$, when $V_k^{a,b} \sqsubset s$, we denote by $v_k^s$ the (necessarily unique) situation in $V_k^{a,b}$such that $v_k^s \sqsubset s$.

	 For any situation~$s$ on a path $\omega \in \samplespace$ satisfying (\ref{Eq: levy conv to infty}) we now have one of the following cases:
	 \begin{enumerate}
	 \item
	 The first case is that $s \in [ \Box , V_1^{a,b}]$. Then we have 
	 \begin{equation*}
	 \mathscr{T}^{a,b} (s) = \mathscr{T}^{a,b} (\Box) = 1.
	 \end{equation*}
	  \item
	 The second case is that $s \in (V_k^{a,b} , U_{k}^{a,b}]$ for some $k \in \nats$. Then we have
	 \begin{equation*}
	 \mathscr{T}^{a,b} (s) = \Bigg( \prod_{\ell = 1}^{k-1} \frac{\martingale_{\ell}^{a,b}(u_\ell^{s})}{\martingale_{\ell}^{a,b}(v_{\ell}^{s})} \Bigg) \frac{\martingale_{k}^{a,b}(s)}{\martingale_{k}^{a,b}(v_{k}^{s})}.
	 \end{equation*} 
	 Since $\martingale{}_k^{a,b}(s) \geq \inf f > 0$ and, for all $\ell \in \{ 1,...,k\}$, $\martingale_{\ell}^{a,b}(u_\ell^{s}) > b > 0$ and $0 < \martingale_{\ell}^{a,b}(v_\ell^{s}) < a$, we get
	 \begin{equation*}
	 \mathscr{T}^{a,b} (s) > \Big( \frac{b}{a} \Big)^{k-1}  \frac{\martingale_{k}^{a,b}(s)}{a} \geq \Big( \frac{b}{a} \Big)^{k-1} \Big( \frac{\inf f}{a} \Big). 
	 \end{equation*} 
	 \item
	 The third case is that $s \in (U_k^{a,b} , V_{k+1}^{a,b}]$ for some $k \in \nats$. Then we have
	 \begin{equation*}
	 \mathscr{T}^{a,b} (s) = \prod_{\ell = 1}^{k} \frac{\martingale_{\ell}^{a,b}(u_\ell^{s})}{\martingale_{\ell}^{a,b}(v_{\ell}^{s})}.
	 \end{equation*} 
	 Again, as for all $\ell \in \{ 1,...,k\}$, $\martingale_{\ell-1}^{a,b}(u_\ell^{s}) > b > 0$ and $0 < \martingale_{\ell}^{a,b}(v_\ell^{s}) < a$, we get
	 \begin{equation*}
	 \mathscr{T}^{a,b} (s) > \Big( \frac{b}{a} \Big)^{k}. 
	 \end{equation*} 
	\end{enumerate}
	Because $\inf f > 0$ and $\frac{b}{a} > 1$, and because $\omega$ goes through all the cuts, we conclude that indeed $\lim_{n \to +\infty} \mathscr{T}^{a,b} (\omega^n) = + \infty$.

	To finish, we use the countable set of rational couples $K \coloneqq \{ (a,b) \in \mathbb{Q}^2 : 0<a<b \}$ to define the process $\mathscr{T}$:
	\begin{equation*}
		\mathscr{T} \coloneqq \sum_{(a,b) \in K} w^{a,b} \mathscr{T}^{a,b},
	\end{equation*}
	with coefficients $w^{a,b}>0$ that sum to $1$. 
	Hence, for all $s \in \situations$, $\mathscr{T}(s)$ is a countable convex combination of the real numbers $\mathscr{T}^{a,b}(s)$. 

	By Lemma \ref{lemma: Convex sum}, $\mathscr{T}$ is then also a positive test supermartingale such that $\mu\mathscr{T}(s) \leq \alpha$ for all $s \in \situations{}$.
	Moreover, $\mathscr{T}$ converges to $+ \infty$ on the paths $\omega \in \samplespace{}$ where $\liminf_{n \to +\infty} \upprevvovk(f \vert \omega^n) < f(\omega)$. 
	Indeed, consider such a path $\omega \in \samplespace{}$. 
	Then since $f(\omega) \geq \inf f > 0$, there is at least one couple $(a',b') \in K$ such that $\liminf_{n \to +\infty} \upprevvovk(f \vert \omega^n) < a' < b' < f(\omega)$, and as a consequence $\lim_{n \to +\infty} \mathscr{T}^{a',b'}(\omega^n) = + \infty$. 
	Then also $\lim_{n \to +\infty} w^{a',b'} \mathscr{T}^{a',b'}(\omega^n) = + \infty$ since $w^{a',b'} > 0$. 
	For all other couples $(a,b) \in K \setminus (a',b')$, we have $w^{a,b} \mathscr{T}^{a,b} > 0$, so $\mathscr{T}$ indeed converges to $+ \infty$ on $\omega$.
	\qed
	\end{proof}

\subsection{Upward Convergence and Fatou's Lemma}

We recall the following two properties from the literature.
\begin{lemma}[\protect{\cite[Corollary 3]{DECOOMAN201618}}]\label{lemma2}
Consider any $n \in \natz$, any situation $x_{1:n} \in \situations$, and any $(n+1)$-measurable gamble $h$ on $\Omega$. 
For any $x_{1:n} \in \statespace_{1:n}$, consider the local gamble $h(x_{1:n} X_{n+1})$ on $\statespace{}$, whose value in $x_{n+1} \in \statespace{}$ is given by $h(x_{1:n+1})$ for all $x_{n+1} \in \statespace{}$. 
Then
\begin{equation*}
		\upprevvovk (h \vert x_{1:n}) = \lupprev (h(x_{1:n} X_{n+1}) \vert x_{1:n}).
\end{equation*}
\end{lemma}

\begin{theorem}[\protect{\cite[Theorem 28]{8535240}}, Law of Iterated Expectations]\label{theoremIteratedExpectations}
Consider any extended real variable $g \in \setofextvariables$ and, for all $n \in \natz{}$, the $n$-measurable extended real variable $\upprevvovk(g \vert X_{1:n})$ whose value in $x_{1:n} \in \statespace{}_{1:n}$ is given by $\upprevvovk(g \vert x_{1:n})$.
Then for any $k,\ell \in \natz$ such that $k \leq \ell$, it holds that
\begin{align*}
\upprevvovk(g \vert X_{1:k}) = \upprevvovk(\upprevvovk(g \vert X_{1:\ell}) \vert X_{1:k}).
\end{align*}
\end{theorem}

\begin{lemma}\label{lemmaMonotoneConvergence}
Consider any extended real variable $f \in \setofextvariables$ and situation~$s \in \situations$. 
If $\upprevvovk(f \vert s) < +\infty$, then $\upprevvovk(f \vert t) < +\infty$ for all $t \sqsupseteq s$.
\end{lemma}
\begin{proof}
Fix any $t \sqsupseteq s$.
If $\upprevvovk(f \vert s) < +\infty$, then there is some bounded below supermartingale $\martingale$ such that $\liminf \martingale \geq_s f$. 
But then also $\martingale(t) < +\infty$ and $\liminf \martingale \geq_t f$. 
Hence, $\upprevvovk(f \vert t) \leq \martingale(t) < +\infty$. 
\qed
\end{proof}

\begin{lemma}\label{lemmaMonotoneConvergenceSAS}
Consider any extended real gamble $f \in \setofextvariables$, then
\begin{align*}
	\upprevvovk(f) &= \inf \Big\{ \martingale(\Box) :  \martingale \in \setofsupmartb \text{ and } \liminf \martingale \geq f \text{ s.a.s.} \Big\}.
\end{align*}
\end{lemma}
\begin{proof}
Since every bounded below supermartingale $\martingale$ that satisfies $\liminf \martingale \geq f$ also satisfies $\liminf \martingale \geq f$ s.a.s., we clearly have 
\begin{equation*}
\upprevvovk(f) \geq \inf \Big\{ \martingale(\Box) \colon  \martingale \in \setofsupmartb \text{ and } \liminf \martingale \geq f \text{ s.a.s.} \Big\},
\end{equation*}
so it remains to prove the other inequality.
Fix any $\alpha \in \reals{}$ such that $\alpha > \inf \big\{\martingale{}(\Box) \colon \martingale \in \setofsupmartb \text{ and } \liminf \martingale \geq f \text{ s.a.s.} \big\}$ and any $\epsilon > 0$.
Then there is some bounded below supermartingale $\martingale_\alpha$ such that $\liminf \martingale_\alpha \geq f \text{ s.a.s.}$ and 
\begin{equation}\label{Eq: lemmaMonotoneConvergenceSAS 1}
\martingale_\alpha(\Box) \leq \alpha.
\end{equation}
Since $\liminf \martingale{}_\alpha \geq f$ s.a.s., there is some test supermartingale $\martingale_\alpha^\ast$ that converges to $+ \infty$ on $A \coloneqq \{ \omega \in \samplespace \colon \liminf \martingale_\alpha(\omega) < f(\omega)\}$.
Consider the real process $\martingale_\alpha + \epsilon \martingale_\alpha^\ast$. 
This process is again a supermartingale because of coherence, and it is bounded below because $\martingale_\alpha$ is bounded below and $\epsilon \martingale_\alpha^\ast$ is non-negative. 
Since $\martingale_\alpha^\ast$ converges to $+ \infty$ on $A$ and because $\martingale_\alpha$ is bounded below, we have $\liminf (\martingale_\alpha + \epsilon \martingale_\alpha^\ast)(\omega) = +\infty \geq f(\omega)$ for all $\omega \in A$. 
Moreover, for all $\omega \in \samplespace{} \setminus A$, we also have that $\liminf (\martingale_\alpha + \epsilon \martingale_\alpha^\ast)(\omega) \geq f(\omega)$, because $\liminf \martingale{}_\alpha(\omega) \geq f(\omega)$ and because $\epsilon \martingale_\alpha^\ast$ is non-negative.
Hence $\liminf (\martingale_\alpha + \epsilon \martingale_\alpha^\ast) \geq f$, and consequently $\upprevvovk(f) \leq (\martingale_\alpha + \epsilon \martingale_\alpha^\ast)(\Box)$.
It therefore follows from \eqref{Eq: lemmaMonotoneConvergenceSAS 1} that
\begin{align*}
\upprevvovk(f)
&\leq (\martingale_\alpha + \epsilon \martingale_\alpha^\ast)(\Box) 
= \martingale_\alpha(\Box) + \epsilon \leq \alpha + \epsilon.
\end{align*}
As this holds for any $\epsilon \in \posreals{}$, we have that $\upprevvovk(f) \leq \alpha$, and since this is true for every $\alpha \in \reals{}$ such that $\alpha > \inf \big\{\martingale{}(\Box) \colon \martingale \in \setofsupmartb \text{ and } \liminf \martingale \geq f \text{ s.a.s.} \big\}$, it follows that
\begin{align*}
\upprevvovk(f)
\leq \inf \Big\{ \martingale(\Box) :  \martingale \in \setofsupmartb \text{ and } \liminf \martingale \geq f \text{ s.a.s.} \Big\}. \tag*{\qed}
\end{align*}
\end{proof}



\begin{proposition}[Continuity for Non-Decreasing Sequences of Gambles]\label{proposition: bounded upward convergence game}
Consider an extended real variable $f \in \setofextvariables$  and a non-decreasing sequence $\{f_n\}_{n \in \natz}$ of gambles on $\samplespace{}$ that converges pointwise to $f$. 
If moreover $\upprevvovk(f) < +\infty$, then
\begin{equation*}
	\upprevvovk(f) = \lim_{n \to +\infty} \upprevvovk(f_n).
\end{equation*} 
\end{proposition}
\begin{proof}
As the sequence $\{f_n\}_{n \in \natz}$ is non-decreasing and $f_0$ is a gamble and therefore bounded, there is an $M \in \reals{}$ such that $f_n \geq f_0 \geq M$ for all $n \in \natz{}$ and therefore, $f$ is also bounded below by $M$. 
Hence, since $\upprevvovk$ is constant additive [\ref{coherence 6}*], we can assume without loss of generality that $f$ and all $f_n$ are non-negative.

Since $f \geq f_{n+1} \geq f_n$ for all $n \in \natz{}$, it follows from \ref{coherence 4}* and \ref{coherence 5}* that $\upprevvovk{}(f) \geq \upprevvovk{}(f_{n+1}) \geq \upprevvovk{}(f_n) \geq \inf f_n \geq 0$.
Since $\upprevvovk{}(f) < +\infty$, this tells us on the one hand that $\{\upprevvovk{}(f_n)\}_{n \in \natz{}}$ is a non-decreasing sequence of real numbers that is bounded above, so $\lim_{n \to +\infty} \upprevvovk(f_n)$ exists and is real.
On the other hand, this tells us that $\upprevvovk(f) \geq \lim_{n \to +\infty} \upprevvovk(f_n)$. So it only remains to prove the converse inequality.

For any $n \in \natz$, consider the process $S_n$, defined by $S_n(s) \coloneqq \upprevvovk(f_n \vert s)$ for all $s \in \situations$ and the process $S$ defined by the limit $S(s) \coloneqq \lim_{n \to +\infty}S_n(s)$ for all $s \in \situations$.
This limit exists because $\{S_n(s)\}_{n \in \natz}$ is a non-decreasing sequence for all $s \in \situations$, due to the monotonicity [\ref{coherence 4}*] of $\upprevvovk{}$.
As $f_n$ is non-negative for all $n \in \natz{}$, $S_n$ is non-negative for all $n \in \natz{}$ because of \ref{coherence 5}* and therefore $S$ is also non-negative. 
Since we already know that $\upprevvovk{}(f)$ and $\upprevvovk{}(f_n)$ are real, $S$ and $S_n$ are real processes because of Lemma \ref{lemmaMonotoneConvergence}, for all $n \in \natz{}$.
As a result, $S$ and all $S_n$ are non-negative real processes.

It now suffices to prove that $S$ is a bounded below supermartingale such that $\liminf S \geq f$ strictly almost surely because it will then follow from Lemma \ref{lemmaMonotoneConvergenceSAS} that
\begin{align*}
	\upprevvovk(f) 
	&= \inf \Big\{ \martingale(\Box) :  \martingale \in \setofsupmartb \text{ and } \liminf \martingale \geq f \text{ s.a.s.} \Big\} \\
	&\leq S(\Box) = \lim_{n \to +\infty}\upprevvovk(f_n).
\end{align*}
This is what we now set out to do.

$S$ is bounded below because it is non-negative.
We first show that, for any $n \in \natz$, $S_n$ is a supermartingale.
Fix any $x_{1:k} \in \situations{}$ and consider the local gamble $\upprevvovk(f_n \vert x_{1:k} X_{k+1})$ on $\statespace{}$ as defined in Lemma~\ref{lemma2}. 
Then
\begin{align*}
\lupprev(\Delta S_n \vert x_{1:k}) 
&= \lupprev\big(\upprevvovk(f_n \vert x_{1:k} X_{k+1}) - \upprevvovk(f_n \vert x_{1:k}) \ \big\vert \ x_{1:k} \big) \\
&= \lupprev\big(\upprevvovk(f_n \vert x_{1:k} X_{k+1}) \big\vert x_{1:k} \big) - \upprevvovk(f_n \vert x_{1:k}) \\
&= \upprevvovk\big(\upprevvovk(f_n \vert  X_{1:k+1}) \big\vert x_{1:k} \big) - \upprevvovk(f_n \vert x_{1:k}) \\
&= \upprevvovk(f_n \vert x_{1:k}) - \upprevvovk(f_n \vert x_{1:k}) =0 \ \text{ for any $x_{1:k} \in \situations$},
\end{align*}
where the second equality follows from the constant additivity [\ref{coherence 6}] of $\lupprev{}( \cdot \vert x_{1:k})$, the third equality follows from Lemma \ref{lemma2} and the fourth from the Law of Iterated Upper Expectations (Theorem \ref{theoremIteratedExpectations}). 
It follows that $S_n$ is an supermartingale.

We show next that $S$ is also a supermartingale. 
By definition, $\Delta S_n(s)$ converges pointwise to $\Delta S(s)$ for all $s \in \situations$.
Moreover, $\statespace$ is finite and $\Delta S(s)$ is also a gamble on $\statespace{}$.  
Hence, the sequence $\Delta S_n(s)$ converges \emph{uniformly} to $\Delta S(s)$ for all $s \in \situations$. 
Therefore, for all $s \in \situations{}$, it follows from \ref{coherence 7} that $\lim_{n \to +\infty} \lupprev(\Delta S_n(s) \vert s) = \lupprev(\Delta S(s) \vert s)$, which, since $\lupprev(\Delta S_n(s) \vert s) = 0$ for all $n \in \natz$, implies that $\lupprev(\Delta S(s) \vert s) = 0$. 
Hence, $S$ is a supermartingale.

To prove that $\liminf S \geq f$ strictly almost surely, we will use our version of L\'evy's zero-one law.
For all $n \in \natz{}$, since $f_n$ is bounded, it follows from Theorem~\ref{Bounded Levy} that there is a positive test supermartingale $\martingale{}_n$ that converges to $+\infty$ on the event 
\begin{align*} \label{eq: proof upward convergence eq1}
	A_n \coloneqq \Big\{ \omega \in \samplespace{} \colon \liminf_{m \to +\infty} \upprevvovk(f_n \vert \omega^m) &< f_n(\omega) \Big\}.
\end{align*}
Moreover, if we fix some $\alpha > 1$, the positive test supermartingales $\martingale{}_n$ can be chosen such that $\mu\martingale{}_n(s) \leq \alpha$ for all $s \in \situations{}$ and all $n \in \natz$.
Now, consider the real process $\martingale{}$, defined by
\begin{equation*}
	\martingale{}(s) \coloneqq \sum_{n \in \nats{}} \lambda_n \martingale{}_n(s) \text{ for all } s \in \situations{},
\end{equation*}
where the coefficients $\lambda_n > 0$ sum to $1$.
Then, by Lemma~\ref{lemma: Convex sum}, $\martingale{}$ is also a positive test supermartingale, which furthermore clearly converges to $+\infty$ on $\cup_{n \in \natz{}} A_n \eqqcolon A$.
Consider now any path $\omega \in \samplespace{}$ for which $\liminf_{m \to +\infty} S(\omega^m) < f (\omega)$.
Since, as we explained before, $S_n(s)$ is non-decreasing in $n$ for all $x\in\situations{}$, we have that, for all $m\in\natz$, $\sup_{n \in \natz} S_n(\omega^m) = \lim_{n \to +\infty} S_n(\omega^m) = S(\omega^m)$. Since $\liminf_{m \to +\infty} S(\omega^m) < f (\omega)$, this implies that
\begin{align*}
 \liminf_{m \to +\infty} \sup_{n \in \natz} S_n(\omega^m) 
 < \lim_{n \to +\infty} f_n(\omega).
\end{align*}
Since also $\sup_{n \in \natz} \liminf_{m \to +\infty} S_n(\omega^m) \leq \liminf_{m \to +\infty} \sup_{n \in \natz} S_n(\omega^m)$, we infer that
\begin{align}\label{eq: proof upward convergence eq2}
	\sup_{n \in \natz} \liminf_{m \to +\infty} \upprevvovk(f_n \vert \omega^m) &< \lim_{n \to +\infty} f_n(\omega).
\end{align}
Then there is some $n_\omega$ such that
\begin{equation*}
\sup_{n \in \natz} \liminf_{m \to +\infty} \upprevvovk(f_n \vert \omega^m) <  f_{n_\omega}(\omega),
\end{equation*}
and therefore, we see that also
\begin{equation*}
\liminf_{m \to +\infty} \upprevvovk(f_{n_\omega} \vert \omega^m) <  f_{n_\omega}(\omega).
\end{equation*}
So $\omega \in A_{n_\omega} \subseteq A$ and, as a consequence, $\martingale{}$ converges to $+\infty$ on $\omega$.
Hence, the test supermartingale $\martingale{}$ converges to $+\infty$ on all paths $\omega \in \samplespace{}$ such that $\liminf_{m \to +\infty} S(\omega^m) < f (\omega)$, and therefore $\liminf S \geq f$ holds strictly almost surely.
\qed
\end{proof}

Next, we show that Proposition~\ref{proposition: bounded upward convergence game} also holds for sequences of extended real variables that are bounded below.

\begin{proof}[Theorem~\ref{theorem: upward convergence game}]
That $\upprevvovk{}(f) \geq \lim_{n \to +\infty} \upprevvovk{}(f_n)$ follows directly from the monotonicity [\ref{coherence 4}*] of $\upprevvovk{}$.
We prove the converse inequality.
Consider the sequence of gambles $\{f_n^\ast\}_{n \in \natz{}}$ defined by
\begin{align*}
f_n^\ast(\omega) \coloneqq \min \{f_n(\omega),  n \} \text{ for all } \omega \in \samplespace{}.
\end{align*}
Then it is clear that $\{f_n^\ast\}_{n \in \natz{}}$ is also a non-decreasing sequence that converges pointwise to the extended real variable $f$.
Moreover, since $\{f_n\}_{n \in \natz{}}$ is bounded below, $\{f_n^\ast\}_{n \in \natz{}}$ is also bounded below, and, hence, since each $f_n^\ast$ is bounded above by $n$, $\{f_n^\ast\}_{n \in \natz{}}$ is a sequence of gambles. 
Then, by Proposition~\ref{proposition: bounded upward convergence game}, we have that
\begin{equation*}
	\upprevvovk(f) = \lim_{n \to +\infty} \upprevvovk(f_n^\ast),
\end{equation*} 
which, since $f_n^\ast \leq f_n$ and therefore---since $\upprevvovk{}$ is monotone [\ref{coherence 4}*]---also $\upprevvovk(f_n^\ast) \leq \upprevvovk(f_n)$ for all $n \in \natz{}$, implies that 
\begin{equation*}\label{eq: proof of monotone convergence theorem}
	\upprevvovk(f) \leq \lim_{n \to +\infty} \upprevvovk(f_n). \tag*{\qed}
\end{equation*}
\end{proof}

\begin{proof}[Theorem~\ref{Fatou}]
Consider the variable $g_k$ defined by $g_k(\omega) \coloneqq \inf_{n \geq k}{f_n}(\omega)$ for any $k \in \natz$ and all $\omega \in \samplespace$.
Then $f = \lim_{k \to +\infty}{g_k}$. 
Furthermore, $\{g_k\}_{k \in \natz{}}$ is a non-decreasing sequence of extended real variables that is bounded below, because $\{f_n\}_{n \in \natz{}}$ is bounded below.
Hence, we can use Theorem \ref{theorem: upward convergence game} to find that
\begin{equation*}
\upprevvovk(f) = \lim_{k \to +\infty} \upprevvovk(g_k) = \liminf_{k \to +\infty} \upprevvovk(g_k) \leq  \liminf_{k \to +\infty} \upprevvovk(f_k),
\end{equation*}
where the inequality follows because, for all $k \in \natz{}$, $g_k \leq f_k$ and therefore, since $\upprevvovk$ is monotone [\ref{coherence 4}*], also $\upprevvovk{}(g_k) \leq \upprevvovk{}(f_k)$.
\qed
\end{proof}

\section{Proofs of the results in Section \ref{sec: Continuity with respect to cuts}}\label{sec: Proofs of continuity cuts}

\begin{lemma}\label{Lemma: Lemma Bounded Below Gamble}
Consider any $\alpha \in \reals{}$ and any extended real variable $f \in \setofextvariables{}$. 
Moreover, for every $A \in \reals{}$, let $f_A$ be the variable defined by 
\begin{align*}
f_A(\omega) \coloneqq \max\{f(\omega),A\} \text{ for all } \omega \in \samplespace.
\end{align*}
If there is some bounded below supermartingale $\martingale \in \setofsupmartb$ such that $\martingale(\Box) \leq \alpha$ and moreover $\liminf \martingale \geq f$, then there is some $A^\ast \in \reals{}$ such that $\upprevvovk{}(f_A) \leq \alpha$ for all $A \leq A^\ast$.
\end{lemma}
\begin{proof}
	Since $\martingale$ is bounded below, it immediately follows that there is some $A^\ast \in \reals{}$ such that $\liminf \martingale \geq A^\ast$, and hence also $\liminf \martingale \geq A$ for all $A \leq A^\ast$.
	Fix any such $A \leq A^\ast$.
	Then it is clear that moreover $\liminf \martingale \geq f_A$ and it follows from the definition of $\upprevvovk(f_A)$ that $\upprevvovk(f_A) \leq \martingale(\Box) \leq \alpha$.
	\qed
\end{proof}

\begin{proof}[Theorem~\ref{Lemma Continuity w.r.t. lower cuts}]
$\upprevvovk(f_A)$ is non-decreasing in $A$ because $f_A$ is non-decreasing in $A$ and because $\upprevvovk{}$ is monotone [\ref{coherence 4}*], and therefore $\lim_{A \to -\infty} \upprevvovk(f_A)$ exists.
That $\lim_{A \to -\infty} \upprevvovk(f_A) \geq \upprevvovk(f)$ follows directly from the monotonicity [\ref{coherence 4}*] of $\upprevvovk{}$ and the fact that, for all $A \in \reals{}$, $f \leq f_A$.
It therefore only remains to prove the converse inequality.

If $\upprevvovk(f) = +\infty$, then $\lim_{A \to -\infty} \upprevvovk(f_A) \leq \upprevvovk(f)$ holds trivially.
If $\upprevvovk(f) < +\infty$, fix any $\alpha > \upprevvovk{}(f)$.
Then it follows from the definition of $\upprevvovk(f)$ that there is some bounded below supermartingale $\martingale \in \setofsupmartb$ such that $\martingale(\Box) \leq \alpha$ and $\liminf \martingale \geq f$. 
Lemma \ref{Lemma: Lemma Bounded Below Gamble} now guarantees that there is some $A^\ast \in \reals{}$ such that $\upprevvovk(f_A) \leq \alpha$, for all $A \leq A^\ast$.
Consequently, we also have that $\lim_{A \to -\infty} \upprevvovk(f_A) \leq \alpha$, and since this holds for any $\alpha > \upprevvovk{}(f)$, we conclude that indeed $\lim_{A \to -\infty} \upprevvovk(f_A) \leq \upprevvovk{}(f)$.
\qed
\end{proof}

\begin{proof}[Theorem~\ref{theorem Continuity w.r.t. cuts}]
Since $\upprevvovk{}(f) < +\infty$, there is some bounded below supermartingale $\martingale \in \setofsupmartb$ such that $\liminf \martingale \geq f$.
Of course, $\martingale(\Box)$ is real, and therefore, if we define the extended real variables $f_A$, for all $A \in \reals{}$, by $f_A(\omega) \coloneqq \max\{ f(\omega), A \}$ for all $\omega \in \samplespace{}$, it follows from Lemma~\ref{Lemma: Lemma Bounded Below Gamble} that there is some $A^\ast \in \reals{}$ such that $\upprevvovk{}(f_A) < +\infty$ for all $A \leq A^\ast$.
Consider any such $A \leq A^\ast$ and any non-decreasing sequence of reals $\{B_n\}_{n \in \natz{}}$ such that $A \leq B_0$ and $\lim_{n \to +\infty} B_n = +\infty$.
Then $\{f_{A,B_n}\}_{n \in \natz{}}$ is a non-decreasing sequence of gambles that converges pointwise to the extended real variable $f_A$, for which $\upprevvovk{}(f_A) < +\infty$.
Hence, we can use Proposition~\ref{proposition: bounded upward convergence game} to derive that $\lim_{n \to +\infty} \upprevvovk{}(f_{A,B_n}) = \upprevvovk{}(f_A)$, and therefore also 
\begin{align*}
\lim_{B \to +\infty} \upprevvovk{}(f_{A,B}) = \upprevvovk{}(f_A).
\end{align*}
Since this holds for any $A \in \reals{}$, and because $\lim_{A \to -\infty} \upprevvovk{}(f_A) = \upprevvovk{}(f)$ by Theorem~\ref{Lemma Continuity w.r.t. lower cuts}, we conclude that indeed $\lim_{A \to -\infty} \lim_{B \to +\infty} \upprevvovk{}(f_{A,B}) = \upprevvovk{}(f)$. 
\qed
\end{proof}

 }{}

\end{document}